\documentclass[11pt]{amsart}
   
\usepackage[T1]{fontenc}
\usepackage{amsmath,amsfonts}
\usepackage{amssymb}
\usepackage{amscd}
\usepackage{amsthm}
\usepackage{yhmath}
\usepackage{subfigure}
\usepackage{graphicx}  
\usepackage{mathtools} 
\usepackage{enumerate}

\usepackage[all]{xy}
\usepackage{color}

\usepackage{marginnote}
\usepackage[hidelinks]{hyperref}
\usepackage{verbatim}

\usepackage{todonotes}

\numberwithin{equation}{section}

\newtheorem{theorem}[equation]{Theorem}
\newtheorem{thm}[equation]{Theorem}
\newtheorem{proposition}[equation]{Proposition}
\newtheorem{prop}[equation]{Proposition}

\newtheorem{lemma}[equation]{Lemma}

\newtheorem{corollary}[equation]{Corollary}

\newtheorem{conjecture}[equation]{Conjecture}

\theoremstyle{remark}

\theoremstyle{definition}

\newtheorem{remark}[equation]{Remark}

\newtheorem{definition}[equation]{Definition}

\newtheorem{defn}[equation]{Definition}

\def\XXint#1#2#3{{\setbox0=\hbox{$#1{#2#3}{\int}$}
	\vcenter{\hbox{$#2#3$}}\kern-.5\wd0}}

\newcommand{\cB}{\mathcal{B}}

\newcommand{\N}{\mathbb N}

\newcommand{\R}{\mathbb R}

\newcommand{\Q}{\mathbb Q}

\newcommand{\g}{\mathfrak{g}}

\newcommand{\Lie}{\operatorname{Lie}}

\def\eps{\epsilon}

\newcommand{\vol}{\operatorname{vol}}

\newcommand{\sus}{\subseteq}
\newcommand{\at}[1]{\raise-.5ex\hbox{\ensuremath|}_{#1}}
\newcommand{\der}{\mathrm{d}}
\newcommand{\ja}{\quad \text{and} \quad}
\newcommand{\inv}[1]{{#1}^{-1}}

\newcommand{\spann}[1]{\mathrm{span}(#1)}

\newcommand{\hq}{G_{\Q}}
\newcommand{\hqs}{\hq^{(s)}}

\begin{document}

\title[Differentiability on infinite-dimensional Carnot groups]{G\^ateaux differentiability on infinite-dimensional Carnot groups}

\author{Enrico Le Donne}
\author{Sean Li}
\author{Terhi Moisala}


\begin{abstract}
	This paper contributes to the generalization of Rademacher's differentiability result for Lipschitz functions when the domain is infinite dimensional and has nonabelian group structure. We introduce the notion of metric scalable groups which are our infinite-dimensional analogues of Carnot groups. The groups in which we will mostly be interested are the ones that admit a dense increasing sequence of (finite-dimensional) Carnot subgroups. In fact, in each of these spaces we show that every Lipschitz function has a point of G\^{a}teaux differentiability. We provide examples and criteria for when such Carnot subgroups exist. The proof of the main theorem follows the work of Aronszajn \cite{Aronszajn1976} and Pansu \cite{Pansu}.
	\vspace{-0.7cm}
\end{abstract}

  
\thanks{E.L.D. and T.M. were partially supported by the Academy of Finland (grant
288501
`\emph{Geometry of subRiemannian groups}')
and by the European Research Council
 (ERC Starting Grant 713998 GeoMeG `\emph{Geometry of Metric Groups}'). T.M. was also supported by the Finnish Cultural Foundation (grant 00170709).
  S.L. was supported by NSF grant DMS-1812879.
}
\date{\today}
\maketitle
\setcounter{tocdepth}{3} 
 
\tableofcontents
  \newpage
  

\section{Introduction}
Rademacher's theorem states that Lipschitz maps from $\R^n$ to $\R$ are differentiable almost everywhere.  This has been generalized in many ways over the past few decades.  In the case of considering domains more general than $\R^n$, there have been two distinct branches.  On the one side, there have been studies of what happens when one keeps the vector-space structure but removes finite dimensionality (and thus the measure).  On the other side, there has been interest in removing the vector-space assumption but preserving the metric-measure structure.  Extensions to more general target spaces have also been considered, but in this paper we will not be interested in this route.

We will focus on extending the theorem to domains that are nonabelian and infinite dimensional. We will concentrate on $\R$-valued functions, although the results will hold for more general targets like RNP Banach spaces.  We now quickly review previous results, discuss the issues present in both branches, and provide some references. 

For the case of an infinite-dimensional Banach space domain $X$, derivatives of a function $f$ are linear transforms $T : X \to \R$.  However, there are two ways this may be interpreted.  A function is G\^ateaux differentiable at $x_0 \in X$ if 
\begin{align*}
  T(v) = \lim_{r \to 0} \frac{f(x_0 + tv) - f(x_0)}{t}.
\end{align*}
Instead, a function is Fr\'echet differentiable at $x_0$ if 
\begin{align*}
  f(x_0 + v) = f(x_0) + T(v) + o(\|v\|) ~\text{as} ~\|v\| \to 0.
\end{align*}
Thus, for G\^ateaux differentiability, the rate of convergence as $r \to 0$ can depend on $v$ whereas it only depends on $\|v\|$ for Fr\'echet differentiability.  G\^ateaux differentiability of $f$ at $p$ is equivalent to $ f $ being differentiable at $p$ on every finite-dimensional affine subspace going through $p$.  Fr\'echet differentiability clearly implies G\^ateaux differentiability, but the opposite does not hold in general.  In fact, Lipschitz functions $f : X \to \R$ always have points of G\^ateaux differentiability whereas they may lack any point of Fr\'echet differentiability \cite{Aronszajn1976}.

In infinite-dimensional settings, one also needs to define ``almost everywhere''.  One can reinterpret Rademacher's theorem as stating that the nondifferentiability points lie in the $\sigma$-ideal of Lebesgue null sets.  Thus, one tries to prove that the nondifferentiability points of Lipschitz functions must lie in some $\sigma$-ideal $\mathcal N$.  To guarantee at least one point of differentiability, the $ \sigma $-ideal $\mathcal N$ should not contain open sets.  Results of this type have been found for both G\^ateaux differentiability and Fr\'echet differentiability although the Fr\'echet differentiability results are far harder and less broad \cite{Aronszajn1976,Mankiewicz,Lindenstrauss-Preiss,Preiss-differentiable}.

When one is considering domains that are not vector spaces, one typically assumes that the space is a metric measure space, which makes ``almost everywhere'' clear.  But resolving what a derivative means becomes trickier, and one requires some additional structure on the domain.  For general metric spaces, one needs a collection of Lipschitz charts---which may not exist---to differentiate our original function $f$ against as was done by Cheeger in \cite{Cheeger}.  Other works expanding on this theory of differentiation include \cite{Bate-Li-2,Cheeger-Kleiner4,Eriksson-Bique,Schioppa}.  In the case of Carnot-group domains, there is a group structure as well as a scaling automorphism.  This allows us to define a derivative as the limit of rescaled difference ratios converging to a homomorphism as was done by Pansu in \cite{Pansu}.

This paper seeks to bridge the infinite-dimensional setting with the Carnot-group setting.  Specifically, we will define infinite-dimensional variants of Carnot groups (which we shall call metric scalable groups) and consider G\^ateaux differentiability in this setting.  We will show that if a metric scalable group $G$ has a dense collection of finite-dimensional Carnot subgroups, then there is a nontrivial $\sigma$-ideal $\mathcal N$ so that the G\^ateaux non-differentiability points of any Lipschitz function $f : G \to \R$ form an element in $\mathcal N$.

We remark that there have been previous studies of infinite-dimensional variants of Carnot groups.  Notably, in \cite{Magnani-Rajala}, the authors defined so-called Banach homogeneous groups and showed that Lipschitz functions from $\R$ to these groups are almost everywhere differentiable in the notion of Pansu.  Metric scalable groups include these groups as special cases, but they also contain other examples.

Our investigations leave open several natural questions.  Most notably, one can ask how small the points of Fr\'echet nondifferentiability of Lipschitz functions are for metric scalable groups.   Even in Banach space domains, this is a very hard problem that depends on fine geometric properties of the norm, and so we leave this problem for the future.  One can also ask if there are infinite-dimensional variants of Cheeger differentiability.  Here, the question becomes more subtle as the differentiability charts are now infinite dimensional and so one needs to know in what Banach space they take value.  Finally, one would like to know when metric scalable groups are generated by its finite-dimensional subgroups.  Specifically, are there geometric properties (geodicity, for example) that tell us when this is the case?


We begin by introducing the notion of a scalable group that is the underlying structure of the metric groups with which we will be concerned.

\begin{definition}[Scalable group] \label{def:scalable:group}
	A  \emph{scalable group} is a pair $ (G,\delta) $, where $ G $ is a topological group and $ \delta \colon \R \times G\to  G $  is a continuous  map such that 
$\delta_\lambda:=\delta(\lambda, \cdot) \in	\mathrm{Aut}(G)$ for all $ \lambda \in \R \setminus \{0 \} $,
$$ \delta_\lambda \circ \delta_\mu =\delta_{\lambda \mu},\qquad\forall\, \lambda,\mu\in \R,$$
	 and $ \delta_0 \equiv e_G $, where $ e_G $ is the identity element of $ G $.
	 	 	\end{definition}
In an obvious way, in the setting of scalable groups, one can consider the notion of scalable subgroups; a subgroup $ H $ of a scalable group $ (G,\delta) $ is called a \emph{scalable subgroup} if $ \delta_\lambda(H) = H $ for all $ \lambda \in \R \setminus \{0\}$. In order to talk about Lipschitz functions, we will endow these groups with metrics that make the dilation automorphisms $\delta_\lambda$ metric scalings in the following sense.

	\begin{definition}[Metric scalable group] \label{def:metric:scalable:group}
	A {\em metric scalable group} is a triple $(G, \delta , d)$ where  $ (G,\delta) $ is a scalable group and $d$ is an admissible left-invariant distance on $G$ such that 
$$d(\delta_t(p),\delta_t(q)) = |t| d(p,q), \qquad \forall\, t\in \R.$$
\end{definition}

\noindent By {\em admissible}, we mean that the metric induces the given topology. 

Every Carnot group naturally has  a structure of a scalable group, where by Carnot group $ G $ we mean a simply connected Lie group whose Lie algebra is equipped with a stratification. The stratification is unique up to an isomorphism, see \cite{LeDonne:Carnot}, and it defines a  family of dilations on $ G $. Such a group can be metrized as a metric scalable group, and the metric is unique up to biLipschitz equivalence. On the contrary, we say that a scalable group $(G,\delta)$ has a Carnot group structure if there exists a Carnot group that is isomorphic to $G$ as a topological group and whose dilations given by the stratification coincide with $\delta$.

Given a group structure with a dilation, we can define derivatives as done by Pansu \cite{Pansu}.  As mentioned before, in the infinite-dimensional case, we need to take care of the distinction between G\^ateaux and Fr\'echet differentiability.  Here, we define G\^ateaux differentiability.

\begin{definition}[G\^ateaux differentiability]
Given two scalable groups $G$ and $H$,
a map $f:G\to H$   is {\em G\^ateaux differentiable} at a point $p\in G$ if, as $\lambda\to 0$,  the maps
$\hat f_{p,\lambda}:= \delta_\frac{1}{\lambda} \circ L_{f( p)}^{-1}\circ f\circ L_{ p} \circ \delta_\lambda$  pointwise converge to
a continuous homomorphism from $G $ to $ H$. We denote this map by $ Df_p $ and it is called the {\em G\^ateaux differential} of $ f $ at point $ p $.
\end{definition}
Notice that if $  Df_p  $ exists, then it is 1-homogeneous in the sense that $ Df_p(\delta_\lambda(u)) = \delta_\lambda( Df_p(u) ) $ for all $ \lambda \in \R $ and $ u\in G $.

We now introduce a notion requiring that our groups, which are possibly infinite dimensional, are generated by finite-dimensional Carnot subgroups. 
This will be needed to show that the $\sigma$-ideal we define later is not trivial.

\begin{definition}[Filtration by Carnot subgroups]
We say that a   scalable group $G$ is {\em filtrated by Carnot subgroups}
if there exists a sequence $(N_m)_m$, $m\in \N$, of scalable subgroups of $G$ such that each $N_m$ has a Carnot group structure, $N_m<N_{m+1}$, and
$G$ is the closure of $\cup_{m\in\N}N_m$.
In this case, we say that 
the sequence $N_m$, $m\in \N$,   is {\em a filtration by Carnot subgroups of the  scalable group $G$}.
\end{definition}

Necessarily, a metric scalable group that admits a filtration  by Carnot subgroups  
is separable.
Note that a complete metric scalable group $ G $ cannot be equal to its filtration  $\cup_{m\in\N}N_m$ unless $ G = N_m $ for some $ m\in \N $. Indeed, each $ N_m \sus G$ is nowhere dense and hence the union $\cup_{m\in\N}N_m$ is of first category in $ G $.
 We now define what it means for a set to be null (that is, it  lies in our $\sigma$-ideal).

\begin{definition}[Filtration-negligible]
Given a  filtration 
  $(N_m)_m$, $m\in \N$,   by Carnot subgroups    of a scalable group $G$, we say that a Borel set $\Omega\subseteq G$ is $(N_m)_m$-{\em negligible} if $\Omega$ is the countable union of Borel sets $\Omega_m$ such that 
  $$\vol_{N_m} ( N_m \cap (g  \Omega_m)  )  = 0,\qquad \forall m\in \N, \forall g\in G,$$
  where $\vol_{N_m} $ denotes any Haar measure on $N_m$.
\end{definition}  

We can now state the main theorem of this paper, which is the following generalization of Aronszajn's differentiability result \cite{Aronszajn1976}, and the one of Pansu \cite{Pansu}.

\begin{theorem} \label{th:main}
Let $G$ be a complete metric scalable group.
If $f:G\to \R$ is a Lipschitz map, then there exists a Borel subset $\Omega\subseteq G$ that is $(N_m)_m$-negligible for every    filtration $(N_m)_{m\in\N}$ by Carnot subgroups of $G$ and such that for every $p\notin \Omega$
the map $f$ is G\^ateaux differentiable at $p$.
\end{theorem}

Be aware that a  scalable group may not admit any filtration (for example, if the group is not separable), in which case the above theorem has no content, e.g., one can take $\Omega=G$. Nonetheless, there are large classes of scalable groups that admit filtrations (see Proposition~\ref{exist:filtration} below for a general criterium and Section~\ref{sec:ex} for more examples). 
However, the first thing to clarify is that, as soon as there is one filtration, the whole scalable group cannot be negligible, as the next proposition states.

\begin{proposition}\label{Prop:no:interior}
If  
  $(N_m)_{m\in\N}$ is a  filtration  by Carnot subgroups    of a complete metric scalable group $G$ and 
  $\Omega\subseteq G$   is a Borel $(N_m)_m$-negligible set, then $\Omega$ has empty interior. 
\end{proposition}

As mentioned before, this allows us to conclude that, for groups admitting at least one filtration by Carnot subgroups, every Lipschitz function $f : G \to \R$ has at least one point of G\^ateaux differentiability.

Finally, we would like to have geometric conditions that tell us when our group admits filtrations by Carnot subgroups.  For a scalable group $G$ define its {\em first layer} as 
$$V_1(G) := \{ p\in G : t\in \R \mapsto \delta_t (p) \text{ is a one-parameter subgroup}\}, $$
where by one-parameter subgroup we mean that for all $ t,s \in \R$,
\[
\delta_{t+s}(p) = \delta_t(p)\delta_s(p).
\]
Note that if $ p\in V_1(G)  $, then $ \delta_r(p) \in V_1(G) $ for all $ r\in \R $, since
	\[
	\delta_{t+s}(\delta_r(p)) = \delta_{t r+s  r}(p) = \delta_{t r}(p)\delta_{s r}(p) = \delta_t(\delta_r(p))\delta_s(\delta_r(p)),
	\]
We say that a set \emph{$ A \sus G $ generates $ G $ as a scalable group} or simply that \emph{$ A $ generates $ G $} if $ G $ is the closure of the group generated by $ \{\delta_t(a) : a\in A, \, t\in \R \} $.  Note that $V_1(G)$ is completely analogous to the generating first layer of a finite-dimensional Carnot group.  Moreover, the following proposition holds.

\begin{proposition}\label{exist:filtration}
 Let $G$ be a scalable group. If $G$ admits a filtration by Carnot subgroups then $ V_1(G) $ generates $G$ as a scalable group. On the contrary, if $G$ is nilpotent, $V_1(G)$ is separable, and  $ V_1(G) $ generates $G$ as a scalable group, then $G$ admits a filtration by Carnot subgroups.
\end{proposition}

 In Section \ref{s:examples}, we give an example of a scalable group that admits filtrations by Carnot subgroups but is not nilpotent.

We begin by proving Proposition \ref{exist:filtration} in Section 2. The crucial observation is that any nilpotent group generated by finitely many elements of $V_1(G)$ has a structure of a Carnot group. In Section 3 we make a closer study of filtration-negligible sets among proving Proposition \ref{Prop:no:interior}. Section 4 is devoted to the proof of Theorem \ref{th:main} and finally in Section 5 we introduce a class of metric scalable groups that admit filtrations by Carnot subgroups.

\section{Carnot groups generated}
\label{s:Carnots:generated}

In connection with Proposition \ref{exist:filtration}, the aim of this section is to prove the following proposition.
\begin{prop}
	\label{p:isCarnotSec2}
	Let $ (G,\delta) $ be a scalable group that is generated by $ x_1, \ldots, x_r \in V_1(G) $, with $r\in \N$. If $ G $ is nilpotent, then it is the scalable group underlying some Carnot group.
\end{prop}
We begin by fixing the notation in Section~\ref{ss:notation}. Analogously to Definition~\ref{def:scalable:group}, one can consider \emph{$ \Q $-scalable groups} for which the dilation automorphism is defined on the rationals: $ \delta \colon \Q \times G \to G $. In Section~\ref{ss:QscalableGroups} we prove that if $ G $ is a nilpotent $ \Q$-scalable group of step $ s $ that is generated by finitely many elements, then $ G^{(s)} $ has a structure of finite-dimensional $ \Q $-vector space. Here $ G^{(s)} $ is the last element of the lower central series of the nilpotent group $ G $. Some of the simple commutator identities that we use are proved in Appendix~\ref{ss:commidentities}.

In Section~\ref{ss:proof:isCarnot} we use the result of Section~\ref{ss:QscalableGroups} to show that under the assumption that $ G $ is a nilpotent scalable group generated by finitely many elements, the last layer $ G^{(s)} $ is a real finite-dimensional topological vector space, and in particular it is locally compact. Consequently, see Theorem~\ref{t:loccpt:quotient}, also $ G $ is locally compact. The proof of Proposition~\ref{p:isCarnotSec2} is concluded by the result of Siebert (Theorem~\ref{t:siebert}), which says that any connected, locally compact, contractible group is a positively graduable Lie group. Namely, we find a graduation $ \bigoplus_{t > 0}V_t $ of Lie$(G) $ such that $ V_1 $ generates Lie$ (G) $, and hence $ \bigoplus_{t > 0}V_t $ is a stratification of $ G $.
\subsection{Notation}
\label{ss:notation}
For a group $ G $ and elements $ g,h\in G $ we define the group commutator by
\[
[g,h] \coloneqq ghg^{-1}h^{-1}.
\]
The elements of lower central series are defined by $ G^{(1)} = G $ and $  G^{(k)} $ is the group generated by $ [G,G^{(k-1)}] $. We say that $ G $ is nilpotent of step $ s $ if $ G^{(s+1)} = \{e\} $ but $ G^{(s)} \neq \{e\} $. Notice that in this case $ G^{(s)}  $ is an abelian subgroup of $ G $. We denote by $ Z(G) $ the center of $ G $.

We follow the terminology of \cite{Khukhro}, and define recursively \emph{commutators of weight} $ k $ for $ k\in \N $ in variables $ x_1,x_2, \ldots $ as formal bracket expressions. The letters $ x_1,x_2 \ldots $ are commutators of length one; inductively, if $ c_1, $  $c_2 $ are commutators of weight $ k_1 $ and $ k_2 $, then $ [c_1,c_2] $ is a commutator of weight $ k_1+k_2 $. We also call the commutator of the form $ [x_1,[x_2,\ldots,[x_{k-1},x_k]\ldots]] $ a \emph{simple commutator} of elements $ x_1,\ldots,x_k $. 

During this section, it is useful to keep in mind the following lemma.
\begin{lemma}[Lemma 3.6(c) in \cite{Khukhro}]
	\label{l:simple:commutators}
	Let $ G $ be a group and $ M\sus G $ a subset of $ G $. If $ M $ generates $ G $ as a group, then $ G^{(k)} $ is generated by simple commutators of weight $ \geq k $ in the elements $ m^{\pm 1} $, $ m\in M $.
\end{lemma}
We also write down the definition of a vector space to ease the discussion later on.
\begin{defn}
	\label{def:qstr} Let $ \mathbb{K} $ be a field.
	A $  \mathbb{K}  $-\emph{vector space} is an abelian group $G $ equipped with an operation $\sigma \colon  \mathbb{K}  \times G \to G $ satisfying
	\begin{enumerate}[(i)]
		\item $ \sigma(q,\sigma(p,g)) = \sigma(qp,g) $
		\item $ \sigma(q,g)\sigma(p,g) = \sigma(q+p,g) $ \label{ii}
		\item $ \sigma(1,g) =g$
		\item $ \sigma(q,g)\sigma(q,h) = \sigma(q,gh) $,
	\end{enumerate}
	for all $ q,p \in  \mathbb{K}  $ and $ g,h\in G $. We denote the map $ \sigma(q,\cdot) $ by $ \sigma_q $.
\end{defn}
\subsection{$ \Q $-scalable groups}
\label{ss:QscalableGroups}
In this section, $ G $ will always denote a nilpotent $ \Q $-scalable group of step $ s $ with dilations $ \delta_t $, generated by $ x_1,\ldots,x_r \in V_1(G) $. We will show that the last element $ G^{(s)} $ of the lower central series admits a structure of finite-dimensional $ \Q $-vector space.
\begin{lemma}
	\label{l:dilation}
	Let $ m\in \N $ and $ 2\leq k \leq s $. If $ y \in G^{(k)}$ is a simple commutator of $ k$ elements of $ V_1(G) $, then $ \delta_m(y) = hy^{m^k} $ for some $ h\in G^{(k+1)} $.
\end{lemma}
\begin{proof}
	Induction on $ k $: if $ k=1 $, we have that $\delta_m(y) =  y^m $ since $t \mapsto \delta_t(y) $ is a one-parameter subgroup. Assume that the claim holds for $ k-1 $ and let $ y\in G^{(k)}$. Now $ y = [x,w] $, where $ x\in  V_1(G) $ and $ w \in G^{(k-1)}$ is a simple commutator of $ k-1 $ elements of $  V_1(G) $. Hence
	\[
	\delta_m(y) = [\delta_m(x),\delta_m(w)] = [x^m,z w^{m^{k-1}}],
	\]
	where $ z\in G^{(k)}$.
	By Lemma~\ref{l:split2} and Corollary~\ref{cor:split}, we get that
	\begin{align*}
	\delta_m(y) = h_1[x^m,z][x^m,w^{m^{k-1}}] =  h_1[x^m,z]h_2 [x,w]^{m m^{k-1}} =h[x,w]^{m^k},
	\end{align*}
	where $ h= h_1[x^m,z]h_2 \in  G^{(k+1)} $.
\end{proof}

\begin{lemma}
	\label{l:Gs:vectorspace}
	The abelian group $ G^{(s)} $ is a $ \Q $-vector space, and the scalar multiplication is given by the map $ \sigma_\frac{n}{m}(z) \coloneqq \delta_m^{-1}(z^{nm^{s-1}}) $ for $ s\geq 2 $. Moreover, if $ z=[x,w]\in G^{(s)} $ with $ x\in  V_1(G) $ and $ w\in G^{(s-1)} $, then $ \sigma_q(z) = [\delta_q(x),w] $.
\end{lemma}
\begin{proof}
	If the step $ s=1 $, the group $G^{(s)} = G $ and the $ \Q $-vector space structure is given by the dilation automorphisms $ \delta \colon \Q \times G \to G $, as the maps $ t\mapsto \delta_t(x_i) $ are one-parameter subgroups.
	
	For step $s \geq 2 $, let first $ z\in G^{(s)} $ be a simple commutator of $ s $ elements of $  V_1(G) $. In particular, $ z=[x,w] $, where $ x\in  V_1(G) $ and $ w\in G^{(s-1)} $ is a simple commutator of $ s-1 $ elements of $  V_1(G) $. Define $ \sigma\colon \Q \times G^{(s)} \to G^{(s)} $ for simple commutators by
	\[
	\sigma_q([x,w]) = [\delta_q(x),w].
	\]
	If $ z $ is a product of simple commutators $ z_1,\ldots,z_k \in G^{(s)}$, we define
	\[
	\sigma_q(z_1\cdots z_k) = \sigma_q(z_1)\cdots \sigma_q(z_k).
	\]
	By Lemma~\ref{l:simple:commutators}, this is enough to define the map $ \sigma $ for all $ z\in G^{(s)} $.
	
	We show next that
	\[
	\sigma_\frac{n}{m}(z) = \delta_m^{-1}(z^{nm^{s-1}}),
	\]
	which proves that the map is well defined.
	Let first $ z=[x,w] $, where $ x\in  V_1(G) $ and $ w\in G^{(s-1)} $ is a simple commutator of $ s-1 $ elements of $  V_1(G) $ and $ q = \frac{n}{m} \in \Q_+$, $ n,m\in \N $. Lemma~\ref{l:dilation} gives us that
	\begin{align*}
	\delta_m(\sigma(q,[x,w])) = [\delta_m(\delta_{n/m}(x)),\delta_m(w)] =  [x^n,hw^{m^{s-1}}] = [x^n,w^{m^{s-1}}],
	\end{align*}
	where $ h\in G^{(s)} \sus Z(G) $. Since $ [x,w]\in Z(G) $ as well, we get by iterating Corollary~\ref{splitlemma} that
	\[
	\delta_m(\sigma_q([x,w])) =[x,w]^{nm^{s-1}} = z^{nm^{s-1}}.
	\]
	If $ q\in \Q_- $, we replace $ x $ by $ \inv{x} $ in the above calculation as $ \delta_{-q}(x)=\delta_q(\inv{x}) $ and use Lemma~\ref{inverse}, which gives
	\[
	[\inv{x},w] = \inv{[x,w]},
	\]
	since now $ [\inv{x},[w,x]] = e_G $.

	If $ z\in G^{(s)} $ is a product of simple commutators $ z_1,\ldots, z_k \in G^{(s)}$, we have
	\begin{align*}
	\delta_m(\sigma_q(z_1\cdots z_k)) &= \delta_m(\sigma_q(z_1))\cdots \delta_m(\sigma_q(z_k)) \\
	&= z_1^{nm^{s-1}}\cdots z_k^{nm^{s-1}} \\
	&= (z_1 \cdots z_k)^{nm^{s-1}} \\
	&= z^{nm^{s-1}}
	\end{align*}
	since $ z_i\in Z(G) $ for all $ i $.
	
	Let finally $ z = [x,w] $ be such that $ x\in  V_1(G) $ and $ w$ is an arbitrary element of $ G^{(s-1)}  $. Then by Lemma~\ref{l:simple:commutators} there exist simple commutators $ v_1,\ldots, v_l $ of length $ s-1 $ such that $ w = v_1\cdots v_l $. By Corollary~\ref{splitlemma} we get
	\begin{align*}
	\sigma_q([x,v_1\cdots v_l]) &= \sigma_q([x,v_1]\cdots[x,v_l])= \sigma_q([x,v_1])\cdots \sigma_q([x,v_l])\\
	&= [\delta_q(x),v_1]\cdots [\delta_q(x),v_l] = [\delta_q(x),v_1\cdots v_l]
	\end{align*}
	
	It remains to check that the map $ \sigma \colon \Q\times G^{(s)} \to G^{(s)} $ satisfies the conditions (i)--(iv) in the Definition~\ref{def:qstr}. Part (iv) is true by construction. The conditions (i) and (iii) follow from the fact that $ \delta \colon (\Q^*,\cdot) \to \text{Aut}(G) $ is a group homomorphism:
	\[
	\delta_{q p} = \delta_q \circ \delta_p \ja \delta_1 = id,
	\]
	so
	\[
	\sigma_q(\sigma_p([x,w])) = [\delta_q\circ\delta_p(x),w]=[\delta_{qp}(x),w] = \sigma_{qp}([x,w])
	\]
	and
	\[
	\sigma_1([x,w]) = [\delta_1(x),w]=[x,w].
	\]
	Part (ii) holds by Corollary~\ref{splitlemma} and since $ t\mapsto \delta_t(x) $ is a one-parameter subgroup	for all $ x\in  V_1(G) $, namely
	\[
	\sigma_{q+p}([x,w]) = [\delta_{q+p}(x),w] = [\delta_q(x)\delta_p(x),w] =[\delta_q(x),w][\delta_p(x),w]= \sigma_q([x,w])\sigma_p([x,w]).
	\] 
	Hence the map $ \sigma $ defines a $ \Q $-vector space structure on $ G^{(s)} $.
\end{proof}
\begin{lemma}
	\label{findim}
	The group $ G^{(s)} $ equipped with the $ \Q $-vector space structure of Lemma~\ref{l:Gs:vectorspace} is finite dimensional.
\end{lemma}
\begin{proof}
	The proof is by induction on step. If step $ s=1 $, we have that $ G = V_1(G) $ is commutative and the set $ \{x_1,\ldots,x_r\} $ is a basis for $ V_1(G)$. Suppose that the claim holds for any $ \Q $-scalable group of step $ s-1 $.  Let $ K \coloneqq G/ G^{(s)} $ and define \[
	\hat{\delta} \colon \Q \times K \to K , \quad \hat{\delta_q}(g G^{(s)})  \coloneqq \delta_q(g) G^{(s)}.
	\] This map is well defined, since $ \delta(G^{(s)}) = G^{(s)}$. Hence the group $ K $ is a $ \Q $-scalable group of step $ s-1 $ and it is generated by $\{ x_1G^{(s)}, \ldots, x_rG^{(s)} \} $. Notice that
	\[
	[x G^{(s)},yG^{(s)}]_K = [x,y]_G G^{(s)}.
	\] Let $ \hat{\sigma} \colon \Q \times K \to K $ be the map from Lemma~\ref{l:Gs:vectorspace}, which makes $ K^{(s-1)}  $ a $ \Q $-vector space. By induction hypothesis, there exists a basis $ \{k_1,\ldots,k_l\} $ of $ K^{(s-1)} $. Let $ \pi \colon G \to K $ be the projection and choose $ u_i \in \inv{\pi}(k_i) \sus G^{(s-1)}$ for all $ 1\leq i \leq l $. We show that the set $ \{[x_i,u_j] \,:\, 1\leq i \leq r, \, 1\leq j \leq l \} $ spans $ G^{(s)} $. Since $ G^{(s)} $ commutes, it is enough to show that $ \{[x_i,u_j]\} $ spans all the elements of the form $ [x,u] $, where $ x\in  V_1(G) $ and $ u\in G^{(s-1)} $.
	
	Fix $ z = [x,u]\in G^{(s)} $ such that $ x\in  V_1(G) $ and $ u\in G^{(s-1)} $. There exist $ q_1,\ldots,q_l \in \Q $, $ q_i = \frac{n_i}{m_i} $, such that
	\begin{align*}
	\pi(u) &= \hat{\sigma}_{q_1}(k_1)\cdots \hat{\sigma}_{q_l}(k_l) \\
	&= \hat{\delta}^{-1}_{m_1}((u_1G^{(s)})^{n_1m_1^{s-2}}) \cdots \hat{\delta}^{-1}_{m_l}((u_lG^{(s)})^{n_lm_l^{s-2}}) \\
	&= {\delta}^{-1}_{m_1}(u_1^{n_1m_1^{s-2}}) \cdots {\delta}^{-1}_{m_l}(u_l^{n_lm_l^{s-2}}) G^{(s)}\\
	&\eqqcolon v G^{(s)}.	
	\end{align*}
	Hence there exists an element $ h\in G^{(s)}\sus Z(G) $ such that $ u = vh $. Therefore
	\begin{align*}
	[x,u] &= [x,vh] = [x,v] \\
	&=[x,{\delta}^{-1}_{m_1}(u_1^{n_1m_1^{s-2}}) \cdots {\delta}^{-1}_{m_l}(u_l^{n_lm_l^{s-2}})] \\
	&=[x,{\delta}^{-1}_{m_1}(u_1^{n_1m_1^{s-2}})] \cdots [x,{\delta}^{-1}_{m_l}(u_l^{n_lm_l^{s-2}})] \\
	&={\delta}^{-1}_{m_1}([x^{m_1},u_1^{n_1m_1^{s-2}}]) \cdots {\delta}^{-1}_{m_l}([x^{m_l},u_l^{n_lm_l^{s-2}}]) \\
	&={\delta}^{-1}_{m_1}([x,u_1]^{n_1m_1^{s-1}}) \cdots {\delta}^{-1}_{m_l}([x,u_l]^{n_lm_l^{s-1}}) \\
	&=\sigma_{q_1}([x,u_1])\cdots \sigma_{q_l}([x,u_l]),
	\end{align*}
	where we used Corollaries~\ref{splitlemma} and~\ref{cor:split}.
	Since $ x\in  V_1(G) $, there exists a $ q\in \Q $ and $ i\in \{1,\ldots,r\} $ such that $ x = \delta_q(x_i) $. Thus, by the second part of Lemma~\ref{l:Gs:vectorspace},
	\begin{align*}
	[x,u] &= \sigma_{q_1}([\delta_q(x_i),u_1])\cdots \sigma_{q_l}([\delta_q(x_i),u_l]) \\
	&=  [\delta_{q_1q}(x_i),u_1] \cdots [\delta_{q_lq}(x_i),u_l] \\
	&= \sigma_{q_1q}([x_i,u_1])\cdots \sigma_{q_lq}([x_i,u_l]).
	\end{align*}
\end{proof}
\subsection{Proof of Proposition~\ref{p:isCarnotSec2}}
\label{ss:proof:isCarnot}
Our first task is to prove that $ G $ is locally compact. To show this, we consider the $ \Q $-scalable subgroup $ \hq $ of $ G $ that by definition is generated as a group by $ \{\delta_t(x_i) : t\in \Q, 1\leq i \leq r\}\eqqcolon V_\Q $. Let  $ \sigma \colon \Q \times \hqs \to \hqs  $ be the continuous map from Lemma~\ref{findim} which makes $ \hqs $ a $ k $-dimensional $ \Q $-vector space for some $ k\in \N $. We use the following facts about topological groups to show that $ G^{(s)} $ is a finite-dimensional real topological vector space.
\begin{thm}[Theorem 1.22 in \cite{Rudin}]
	\label{wiki}
	A Hausdorff topological vector space is locally compact if and only if it is finite dimensional.
\end{thm}
\begin{lemma}
	\label{closedgp}
	Every locally compact subgroup of a topological group is closed.
\end{lemma}
\begin{proof}
	This proof is adapted from a Mathematics Stack Exchange post by Eric Wofsey \cite{EricWofsey}. Let $ H $ be a topological group and let $ K $ be a locally compact subgroup of $ H $. Then $ \overline{K} $ is also a subgroup of $ H $, and $ K $ is dense in $ \overline{K} $. We claim that every locally compact dense subset of a Hausdorff space is open. Indeed, let $ S $ be a locally compact dense subset of a Hausdorff space $ X $ and take $ x \in S $. Let also $ U $ be open in $ S $ such that $x \in  U $, $ \overline{U}\sus S $, and $ \overline{U} $ is compact. Take then an open set $ V \sus X $ such that $ V \cap S = U $. Since $ X $ is Hausdorff, $ \overline{U} $ is closed in $ X $ and therefore $ V \setminus \overline U $ is open in $ X $. But
	\[
	( V \setminus \overline U)\cap S = (V \cap S)\setminus \overline{U} = U \setminus \overline{U} = \emptyset,
	\]
	and hence $ V \setminus \overline U = \emptyset $ as $ S $ is dense in $ X $. We conclude that $ V \sus S $, which proves the claim.
	
	Hence, by the previous claim $ K $ is open in $ \overline{K} $. Recall that every open subgroup of a topological group is closed since the complement $ K^c $ of an open subgroup $ K $ is the union of open sets; $ K^c = \cup_{x \in K^c}xK $. Hence $ K $ is closed in $ \overline{K} $ and therefore also in $ H $.
\end{proof}
\begin{lemma}
	\label{hsvsp}
	$G^{(s)} $ equals to $\overline{\hqs} $ and it is a $ k $-dimensional real topological vector space.
\end{lemma}
\begin{proof}
	Let $ \{v_1,\ldots,v_k \} $ be a basis for $ \hqs $. We claim that since $ \hqs \sus Z(G)$, we may assume that each $ v_i $ is of the form $ [x_i,w_i] $ with $ x_i\in V_\Q $ and $ w_i\in \hq^{(s-1)} $. Indeed, recall that by Lemma~\ref{l:simple:commutators} any element of $ G_\Q^{(s)} $ is a product of simple commutators of elements of $ V_\Q $ of weight $ s $, which proves the claim. Let
	\[
	W\coloneqq\{[\delta_{t_1}(x_1),w_1]\cdots [\delta_{t_k}(x_k),w_k] \,|\, t_i \in \R
	\},
	\]
	which is a group by Corollary~\ref{splitlemma} and since $ t\mapsto \delta_t(x_i) $ is a one-parameter subgroup for each $ i \in \{1,\ldots,k\}$.
	Now $\hqs \sus W $ by definition of $ \sigma $ and $ W \sus \overline{\hqs} $ by continuity of dilations.	We define $ \tilde{\sigma}\colon \R \times W \to W $ by
	\[
	\tilde{\sigma}_\lambda([\delta_{t_1}(x_1),w_1]\cdots [\delta_{t_k}(x_k),w_k]) = [\delta_{\lambda t_1}(x_1),w_1]\cdots [\delta_{\lambda t_k}(x_k),w_k].
	\]
	This map is continuous and it defines an $ \R $-vector space structure on $ W $: since $ \tilde{\sigma} $ is a continuous extension of $ \sigma $, it is easy to show that $ \tilde{\sigma} $ fulfills the conditions in Definition~\ref{def:qstr}. Hence $ W $ is a $ k $-dimensional real topological vector space. Therefore, by Theorem~\ref{wiki} and Lemma~\ref{closedgp}, $ W $ is closed and so $ W = \overline{\hqs} $.
	We conclude the proof by noting that $G =  \overline{\hq} $, and hence $ G^{(s)} = \overline{\hq}^{(s)} = \overline{\hqs} $, where the last equality follows form the continuity of the group operation.
\end{proof}
\begin{thm}[\cite{MZ_1974} p. 52]
	\label{t:loccpt:quotient}
	If a topological group $ G $ has a closed subgroup $ H $ such that $ H $ and the coset-space $ G/H $ are locally compact, then $ G $ is locally compact.
\end{thm}
\begin{lemma}
	\label{l:loccpt}
	Let $ (G,\delta) $ be a nilpotent scalable group that is homogeneously generated by $ x_1, \ldots, x_r \in G $ over $ \R $. Then $ G $ is locally compact.
\end{lemma}
\begin{proof}
	The proof is again by induction on step. If step $ s=1 $, the group $ G $ is a real topological vector space with basis $ \{x_1,\ldots, x_r\} $ and hence locally compact by Theorem~\ref{wiki}. Assume that the claim holds for step $ s-1 $ and consider $ K\coloneqq G/G^{(s)}$, which is generated by $x_1 G^{(s)},\ldots,x_r G^{(s)} $ with dilations $ \hat{\delta}_t(x G^{(s)}) \coloneqq \delta_t(x) G^{(s)}$. Now $ K $ is indeed an $ \R $-scalable topological group, since $ G^{(s)} $ is a closed normal subgroup of $ G $ by Lemma~\ref{hsvsp}. Hence $ K $ is locally compact by the induction hypothesis, and by Theorem~\ref{wiki} the group $ G^{(s)} $ is locally compact as well. Finally Theorem~\ref{t:loccpt:quotient} proves the claim.
\end{proof}
To prove Proposition~\ref{p:isCarnotSec2}, we use the result of Siebert below.
\begin{defn}
	Let $ G $ be a topological group. A continuous automorphism $ \zeta $ of $ G $ is said to be \emph{contractive} if $ \lim_{n\to \infty}\zeta^n(x) = e_G $ for all $ x\in G $. A group that admits a contractive continuous automorphism is called \emph{contractible}.
\end{defn}
\begin{thm}[Corollary 2.4 in \cite{Siebert}]
	\label{t:siebert}
	A topological group $ G $ is a positively graduable Lie group if and only if it is connected, locally compact and contractible. In particular, if $ \zeta \in \mathrm{Aut}(G)$ is contractive, then the graduation $ \bigoplus_{t > 0}V_t  $ given by $ \zeta $ is such that
	\[
	\{ X\in \mathrm{Lie}(G) \,|\, (\der \zeta - \alpha \,\mathrm{id}) X = 0 \} \sus V_{-\ln|\alpha|}
	\]
\end{thm}
\begin{proof}[Proof of Proposition~\ref{p:isCarnotSec2}]
	We proved in Lemma~\ref{l:loccpt} that the group $ G $ is locally compact. It is also connected, since the map $ \gamma_x\colon [0,1]\to G $, $ \gamma_x(t) =  \delta_t(x) $ is a continuous path between $ e_G $ and $ x $ for all $ x\in G $. Additionally, the group $ G $ is contractible as the automorphisms $ \delta_t $ are contractive for all $ t\in (0,1) $; for a fixed $ t\in ( 0,1) $ we have that
	\[
	\lim_{n\to \infty} \delta_t^n(x) = 	\lim_{n\to \infty} \delta_{t^n}(x)  = \delta_0(x) = e_G
	\]
	for all $ x\in G $.
	Hence by Theorem~\ref{t:siebert} the group $ G $ is a Lie group and each $ \delta_t $, $ t\in (0,1) $, defines a positive graduation for Lie$ (G) $. We claim that, in order to prove that $ G $ admits a structure of a Carnot group, it is enough to find a graduation of Lie$(G)  $ such that $ V_1 $ generates the whole of Lie$ (G) $. Indeed, a stratification of a Lie algebra $ \g $ is equivalent to a positive grading whose degree-one layer generates $ \g $ as a Lie algebra.
	
	Let us consider the graduation given by $ \delta_{1/e} $. By Theorem~\ref{t:siebert},
	\[
	\{ X\in \mathrm{Lie}(G) \,|\, (\der \delta_{1/e} - \dfrac{1}{e}\,\mathrm{id}) X = 0 \} \sus V_{-\ln(1/e)} = V_1.
	\]
	Let then $ x\in \{x_1,\ldots,x_r\} $. Since the map $ t\mapsto \delta_t(x) $, $ t\in \R $, is now a one-parameter subgroup of a Lie group, there exists $ X\in $ Lie$ (G) $ such that
	\[
	\delta_t(x) = \exp(t X)
	\]
	for all $ t \in \R $.	Additionally, $ \exp \colon  $ Lie$ (G)\to G $ is a global diffeomorphism, so
	\[
	\log( \delta_t(x)) = tX
	\]
	and on the other hand
	\[
	\log( \delta_t(x)) = \log( \delta_t(\exp(X))) = \log(\exp (\der \delta_t(X))) = \der \delta_t(X).
	\]
	Hence
	\[
	\der \delta_{1/e}(X) =  \dfrac{1}{e}X
	\]
	and $ X \in V_1 $. Therefore $ \log(x_i)\in V_1 $  for all $ i\in \{1,\ldots,r\} $. Notice that, for any $ Y\in \g $, we have for some $ l\in \N $ and $ i_l \in \{1,\ldots,r\} $ that
	\[
	\exp(Y) = \delta_{t_1}(x_{i_1})\cdots \delta_{t_l}(x_{i_l}) = \exp(t_1 \log(x_{i_1}))\cdots \exp(t_l \log(x_{i_l})).
	\]
	Hence $ \{\log(x_1),\ldots, \log(x_r)\} $ generates Lie$ (G) $ as a Lie algebra by the Baker-Campbell-Hausdorff formula and $ V_1 = \spann{\log(x_1),\ldots, \log(x_r)} $. Thus the graduation given by $ \delta_{1/e} $ is a stratification and $ G $ has a structure of a Carnot group.
	
	We still need to verify that the one-parameter family $ (\delta_t)_{t\in \R} $ of Lie group automorphisms are the Carnot group dilations given by the stratification. The Carnot group dilation of factor $ t\neq 0 $ is by definition the unique map $ \zeta_t \in \mathrm{Aut}(G) $ such that
	\begin{equation}
	\tag{$ \star $}
	\label{Carnotdilation}
	\der \zeta_t(X) = t^k X \quad \text{for all }X\in V_k,
	\end{equation}
	and $ \der \zeta_0 $ is the zero map. Obviously $ \der \delta_0 = 0 $, so consider the case $ t \neq 0 $.	Now each $ V_k $ is spanned by simple commutators of $ \log(x_i) $, $ i\in \{1,\ldots,r\} $. Since $ \der \delta_t $ is a Lie algebra homomorphism, we get for these elements that
	\begin{align*}
	\der\delta_t([\log(x_{i_k}),\ldots,[\log(x_{i_{2}}),\log(x_{i_1})]]) &=  [\der \delta_t(\log(x_{i_k})),\ldots,[\der \delta_t(\log(x_{i_{2}})),\der \delta_t(\log(x_{i_1}))]] \\ &= [t \log(x_{i_k}),\ldots,[t \log(x_{i_{2}}),t \log(x_{i_1})]] \\ &= t^k [\log(x_{i_k}),\ldots,[\log(x_{i_{2}}),\log(x_{i_1})]].
	\end{align*}
	By linearity of $ \der \delta_t $ we conclude that the maps $ \der \delta_t $ satisfy condition \eqref{Carnotdilation}. Hence the scalable group $ (G,\delta) $ is a Carnot group and the dilations $ \delta_t $, $t\in \R $, are the unique Carnot group dilations given by the stratification.
\end{proof}

After Proposition~\ref{p:isCarnotSec2}, we shall easily get  Proposition~\ref{exist:filtration}.
\begin{proof}[Proof of Proposition~\ref{exist:filtration}]
Assume first that $(N_m)_m$ is a filtration by Carnot subgroups of $G$. Then the set $\cup_m V_1(N_m) \sus V_1(G)$ and it generates $\cup_m N_m$. As $\cup_m N_m$ is dense in $G$, we get that $V_1(G)$ generates $G$ as a scalable group.  Assume then that $V_1(G)$ generates and let $ (a_n)_{n\in \N}\sus V_1(G) $ be a dense sequence in $ V_1(G) $. This sequence generates a dense subgroup of $ G $, and choosing $ N_m $ to be the scalable group generated by $ \{a_1,\dots, a_m \} $ gives $ G $ a filtration by Carnot groups by Proposition \ref{p:isCarnotSec2}.
\end{proof}

It would be interesting to find geometric conditions that would allow us to conclude that $V_1(G)$ generates $G$.  Indeed, in the case of simply connected nilpotent Lie groups admitting dilations, geodicity then implies that the first layer generates the entire group as rectifiable curves can be approximated by horizontal line segments.  Thus, we make the following conjecture.

\begin{conjecture}
If $G$ is a metric scalable group that is separable and nilpotent and its distance is geodesic, then
 its first layer $V_1(G)$  generates $G$.
\end{conjecture}

%

\section{Negligible sets of metric scalable groups}
In this section $ (G,d,\delta) $ denotes a metric scalable group (according to Definition~\ref{def:metric:scalable:group}).

\subsection{Elementary properties of metric scalable groups}

\begin{lemma} \label{l:isometry}
	 For each $p \in V_1(G)$, the map $t \mapsto \delta_t(p)$ is a homothetic embedding from $\R$ to $G$.
\end{lemma}
\proof
	Let $c = d(0,\delta_1(v)) > 0$.  We claim that $d(\delta_\alpha(v),\delta_\beta(v)) = c|\alpha-\beta|$.  Indeed, as $\delta_t(v)$ is a one parameter subgroup, we get by left-invariance that
	\[
	d(\delta_\alpha(v),\delta_\beta(v)) = d(0,\delta_{\beta - \alpha}(v)) = |\beta - \alpha| d(0,v) = c|\beta - \alpha|.
	\qedhere
	\]

 \begin{lemma} \label{l:compact-perturb}
 	Let $K \subset G$ be a totally bounded set and $\epsilon > 0$.  There exists $\delta > 0$ so that
 	\begin{align*}
 	d(hk,k) < \epsilon, \qquad \forall h \in B(0,\delta), k \in K.
 	\end{align*}
 \end{lemma}
  \begin{proof}
 	As $K$ is totally bounded, there is a finite number of points $\{y_1,...,y_n\} \in G$ so that $K \subseteq \bigcup_{j=1}^n B(y_j,\epsilon/4)$.  Choose $\delta$ small enough so that for any $h \in B(0,\delta)$, we have that $\max_{1 \leq j \leq n} d(y_j,hy_j) < \epsilon/4$.  Now let $k \in K$ and $y_i$ be so that $d(k,y_i) < \epsilon/4$.  Then for any $h \in B(0,\delta)$, we get
 	\begin{align*}
 	d(hk,k) \leq d(hk,hy_i) + d(hy_i,y_i) + d(y_i,k) = 2d(y_i,k) + d(hy_i,y_i) < \epsilon,
 	\end{align*}
 	where we used left-invariance of the metric.
 \end{proof}

\begin{lemma} \label{l:box}
	Let $ G $ be a complete metric scalable group.
  For every $i\in \N$, let  $\psi_i: \R \to G $ be  continuous   such that $\psi_i(0) = e_G$.  Then for every non-empty open set $U$ there exists a sequence of positive numbers $\alpha_1,\alpha_2,... > 0$ so that
  the  map
  \begin{align*}
    \phi : \prod_{i=1}^\infty [0,\alpha_i] &\to G \\
    (t_1,t_2,...) &\mapsto \cdots \psi_2(t_2) \psi_1(t_1)
  \end{align*}
 is well defined and has range in $U$.
\end{lemma}

\begin{proof}
We may assume that $U$   contains the unit ball at $e_G$.
  Note that for each $k \in \N$, $K_k := \phi(\prod_{i=1}^k [0,\alpha_i] \times (0,0,...))$ is a compact set in $G$.  We can construct $\alpha_i$ recursively.  Let $\alpha_1 = 1$ and choose $\alpha_i > 0$ so that
  $$\sup_{g \in K_i} \sup_{t \in [0,\alpha_{i+1}]} d(g, \psi_{i+1}(t)g) < 2^{-i}.$$
  This is possible by Lemma~\ref{l:compact-perturb} and the fact that $\psi_{i+1}(0) = e$ is continuous at 0.  Then each sequence defining a $\phi(t_1,t_2,...)$ is Cauchy and so the limit exists. The fact that the image is in $U$ also follows immediately.
\end{proof}

\subsection{Non-negligibility of open sets: Proof of Proposition \ref{Prop:no:interior}}


%
%

Let $ (N_m)_{m\in\N} $ be
a filtration of $G$ by Carnot groups. 
Assume that a $(N_m)_m$-negligible set $\Omega $  contains an open non-empty set $U$. 

For every $ {m\in\N} $, choose  $  k_m, v_m$ so that
$ \{v_1, \ldots, v_{k_m}\} $ is a basis of $\Lie(N_m)$.


Let  $\psi_m: \R \to G $ be $\psi_m(t):=\exp(t \log(v_m)) $. With the above choice of $U$,
let $ (\alpha_m)  $ and $ \phi $ be as in Lemma~\ref{l:box}.
Notice that the maps $ \tilde{\phi}_m \colon \R^{k_m} \to N_m $, $   \tilde{\phi}_m(t_1,t_2,\ldots, t_{k_m}) := \psi_{k_m}( t_{k_m}) \cdots \psi_2(t_2) \psi_1(t_1)$, are diffeomorphisms.


%
%
%
Let then
\[
 \phi_m \coloneqq \tilde{\phi}_m\at{\prod_{i=1}^{k_m}[0,\alpha_i]}\,.
\]


    Let $\mu$ be the measure on $G$ that is the pushforward via $\phi$ of the probability measure on $\prod_{i=1}^\infty [0,\alpha_i]$.  
    
    On the one hand, Since $\phi$ has image contained in $U$, we have  $\mu(U) = 1$ and hence $\mu(\Omega) = 1$. 
    
    On the other hand, we shall show that  $\mu(\Omega)=0$. Since the set $\Omega$ is
     $(N_m)_m$-negligible, then $\Omega=\cup_{m\in\N}\Omega_m$ for some $ \Omega_m $ such that $\vol_{N_m}(N_m \cap g \Omega_m) = 0$ for all $g \in G$ and $m\in \N$;
and it is enough to show that $\mu(\Omega_m) = 0$.
For doing so, fix $ m \in \N $ and
let $\nu_1$ and $\nu_2$ denote the product probability measures on $C_1 = \prod_{i=1}^{k_m} [0,\alpha_i]$ and $C_2 = \prod_{i=k_m+1}^\infty [0,\alpha_i]$, respectively. 
First, notice that $(\phi_m)_\#(\nu_1)$ is a smooth measure on some open set of $N_m$ and hence it is absolutely continuous with respect to $\vol_{N_m}$.
Second, notice that for any ${\bf t}_2\in C_2$
 the image $\phi(C_1 \times \{{\bf t}_2\}) = \phi(\bar 0,{\bf t}_2)\phi_m(C_1)$ is a positive measure set in $\phi(\bar 0,{\bf t}_2) N_m$.  However, 
   for any ${\bf t}_2\in C_2$ we get
  \begin{eqnarray*}
    \int_{C_1} \chi_{\phi^{-1}(\Omega_m)}({\bf t}_1,{\bf t}_2) ~d\nu_1({\bf t}_1)  
     &   = &\int_{C_1} \chi_{\Omega_m}(\phi(\bar 0,{\bf t}_2)\phi_m({\bf t}_1))~d\nu_1({\bf t}_1)\\
    &= &\int_{C_1} \chi_{\phi_m^{-1}(\phi(\bar 0,{\bf t}_2)^{-1}\Omega_m)}({\bf t}_1) ~d\nu_1({\bf t}_1)\\
      &  = &\nu_1({\phi_m^{-1}(\phi(\bar 0,{\bf t}_2)^{-1}\Omega_m)})\\
        &= & (\phi_m)_\#(\nu_1) ( N_m \cap\phi(\bar 0,{\bf t}_2)^{-1}\Omega_m)\\
 & \preceq & \vol_{N_m}( N_m \cap\phi(\bar 0,{\bf t}_2)^{-1}\Omega_m)=0
  .
  \end{eqnarray*}
  Thus, $\mu(\Omega_m) = \int_{C_2} \int_{C_1} \chi_{\phi^{-1}(\Omega_m)} ~d\nu_1 ~d\nu_2 = 0$.
\qed

\begin{remark}
	Note that the statement of Proposition~\ref{Prop:no:interior} makes sense for   scalable groups without any metrics.  Indeed, the notion of filtrations (and thus also negligibility) only relies on the topology.  Thus, it may be possible that the result is true for all scalable groups although we have not verified this.
\end{remark}


The rest of this section is devoted to a closer study of filtration-negligible sets of metric scalable groups. Below we define an \emph{exceptional class} as in \cite{Aronszajn1976} and prove that it is equivalent to our notion of filtration negligible sets.

\subsection{The exceptional class $\mathcal U$}
Let $G$ be a scalable group with identity element denoted by $e_G$ and $\cB(G)$ be the Borel sets of $G$. For every $a\in V_1(G)$, with $a\neq e_G$ set
$$\mathcal U (a) := \{   A \in \cB(G) :  \forall g\in G, |A  \cap  (g\cdot \R a)| =0   \},$$ 
where we denote by $\R a$ the image of the curve $t\in \R \mapsto \delta_t a$ and by $|\cdot|$ the 1-dimensional Lebesgue measure 
on the curve. In other words, 
$$|A  \cap  (g\cdot \R a)| = | \{ t\in \R :  g \delta_t a  \in A \} |.$$
 For every countable set $\{a_n\}\subset V_1(G)\setminus \{ e_G \}$ set
$$\mathcal U (\{a_n\}) := \{A \in \cB(G) :  A=\cup_{n\in\N} A_n, A_n \in \mathcal U (a_n)    \}.$$ 
Finally, set $\mathcal U$ to be the intersection of all $\mathcal U (\{a_n\})$ among all dense sequences $\{a_n\}\subseteq V_1(G)\setminus \{ e_G \}$.

The classes $\mathcal U (a), \mathcal U (\{a_n\}), \mathcal U $ are $\sigma$-additive, hereditary, and do not contain any open non-empty  set (see the theorem below). Moreover, we have the property:
$$\{a'_n\}\subseteq \{a_n\}\implies \mathcal U (\{a'_n\}) \subseteq \mathcal U (\{a_n\}) .   $$


\begin{thm}
	\label{t:exceptional:iff:negligible}
	Let $ G $ be a metric scalable group and let $\{a_n\}\subset V_1(G)\setminus \{ e_G \}$ be a dense sequence such that the group $ N_m $ generated by $ \{a_1,\ldots, a_m\} $ is nilpotent for all $ m \in \N $. Then a set $ \Omega \sus G $ is in the class $ \mathcal{U}(\{a_n\}) $ if and only if it is $ (N_m)_{m\in \N} $ -negligible.
\end{thm}
Note that by Proposition \ref{p:isCarnotSec2}, each $ N_m $ in the theorem above is a Carnot group and the statement makes sense. The proof of the theorem will be a straightforward consequence of Proposition \ref{p:S-sequence}. The proof of Proposition \ref{p:S-sequence} needs some preparation, and we postpone it to the end of this section.
\begin{lemma}
  Let $A \subset G$ be a bounded Borel set and choose a $v \in V_1(G)$.  Then the function
  \begin{align*}
    f_A(x) = |A \cap (x \cdot \R v)|
  \end{align*}
  is Borel.
\end{lemma}

\begin{proof}
  Let $R > 0$ be arbitrary and $\mathcal{A}$ denote the set of all $A \subseteq B(0,R)$ that satisfy the conclusion.  We will prove $\mathcal{A}$ contains the Borel sets of $B(0,R)$.  We first prove that the open sets in $B(0,R)$ are in $\mathcal{A}$.  Indeed, let $A$ be open and $t \in \R$.  We will show $A' = f_A^{-1}((t,\infty))$ is open.
  
  As $f_A$ is nonnegative, we may suppose without loss of generality that $t \geq 0$.  Let $g \in f_A^{-1}((t,\infty))$ and $\delta = f_A(g) - t > 0$.  Let $E = \{ s\in \R : g \delta_s(v) \in A\}$, which is a bounded set by boundedness of $A$ and Lemma~\ref{l:isometry}.  For each $s \in E$, define $d(s) = d(g \delta_s(v),A^c)$, a positive continuous function on $E$.  We can choose $\epsilon > 0$ small enough so that
  \begin{align*}
    E' = \{s \in \R : d(g \delta_s(v)) > \epsilon\},
  \end{align*}
  satisfies $|E'| > f_A(g) - \frac{\delta}{2}$.

  Note that $E'$ is totally bounded.  By Lemma~\ref{l:isometry}, $\{g \delta_s(v) : s \in E'\}$ is the isometric image of the totally bounded set $E'$ and so also totally bounded.  Thus by Lemma~\ref{l:compact-perturb}, there exists $\eta_0 > 0$ so that
  \begin{align*}
    \sup_{h \in B(0,\eta_0)} \sup_{s \in E'} d(hg\delta_s(v),g\delta_s(v)) < \epsilon.
  \end{align*}
  As $G$ is topological, we have that there exists some $\eta > 0$ small enough so that $B(g,\eta) \subseteq B(0,\eta_0)g$.  This then gives that
  \begin{align*}
    \sup_{h \in B(g,\eta)} \sup_{s \in E'} d(h\delta_s(v),g\delta_s(v)) < \epsilon.
  \end{align*}
  This shows that $h\delta_s(v) \in A$ when $h \in B(g,\eta)$ and $s \in E'$ and so
  $$f_A(h) \geq |E'| > f_A(g) - \frac{\delta}{2} > t,$$
  which proves $B(g,\eta) \subseteq f_A^{-1}((t,\infty))$ and so $f_A^{-1}((t,\infty))$ is open.

  We now show that $\mathcal{A}$ is a monotone class of sets, which will prove that $\mathcal{A}$ contains all Borel sets.  Let $\{E_i\}$ be an increasing sequence in $\mathcal{A}$ and $E = \bigcup_i E_i$.  Then $E \cap (x \cdot \R v) = \bigcup_i (E_i \cap (x \cdot \R v))$, which is also an increasing family and so by monotone convergence theorem we get
  \begin{align*}
    f_E(x) = \lim_{i \to \infty} f_{E_i}(x).
  \end{align*}
  Thus, $f_E$, the increasing pointwise limit of $f_{E_i}$, must be Borel and so $E \in \mathcal{A}$.  Similarly, let $\{E_i\}$ be a decreasing sequence in $\mathcal{A}$ and let $E = \bigcap_i E_i$.  We have as before that $E \cap (x \cdot \R v) = \bigcap_i (E_i \cap (x \cdot \R v))$, another decreasing sequence.  As $E_1$ is bounded, $f_{E_1}(x) < \infty$ and so by dominated convergence theorem we get $f_E(x) = \lim_{i \to \infty} f_{E_i}(x)$.  Thus, as before, $E \in \mathcal{A}$, which proves the monotonicity property of $\mathcal{A}$.
\end{proof}

\begin{lemma} \label{l:borel-decomp}
  Let $A \subseteq G$ be any Borel set and $v \in V_1(G)$.  Then the set
  \begin{align*}
    \{g \in A : |A \cap (g \cdot \R v)| > 0\}
  \end{align*}
  is Borel.
\end{lemma}

\begin{proof}
  Let $A_n = A \cap B(0,n)$.  By monotone convergence theorem, the set in question is equal to $\bigcup_{n=0}^\infty \{g \in A_n : |A_n \cap (g \cdot \R v)| > 0\} = \bigcup_{n=0}^\infty (f^{-1}_{A_n}((0,\infty)) \cap A_n)$, which, by the previous lemma, is a countable union of Borel sets.
\end{proof}

\subsection{Null decomposition}

Let $G$ be a Carnot group with $\dim V_1 = n$ and suppose $G$ is homeomorphic to $\R^m$.  We let $X_1,...,X_n$ be the vector fields in $\R^m$ that are given by left translation of a basis in $V_1$.

\begin{remark}\label{remark:analytic}
  Let $M$ be an analytic manifold in $\R^m$.  Then it is clear that for every $i \in \{1,...,n\}$, the set of points for which $X_i(p) \in T_pM$ is a closed analytic submanifold of $M$.
\end{remark}

\begin{lemma}
  Let $M$ be an analytic manifold in $\R^m$ of dimension less than $m$.  Then for almost every $p \in M$, there exists an open neighborhood $U \ni p$ and an index $i \in \{1,...,n\}$ for which $X_i(q) \notin T_qM$ for all $q \in U$.
\end{lemma}

\begin{proof}
  For each $i \in \{1,...,n\}$, let $A_i = \{p \in M : X_i(p) \in T_pM\}$.  By Remark~\ref{remark:analytic}, we have that $A_i$ are closed analytic submanifolds and $\bigcap_i A_i$ is a closed submanifold where $X_i(p) \in T_pM$ for all $i$.  By non-integrability of $M$ (as $\dim M < m$), we have that $\bigcap_i A_i$ has measure zero.  In particular, there exists $i$ and some $p \in A_i^c$.  As $A_i^c$ is open, there exists an open neighborhood $p \in U \subseteq A_i^c$.  This neighborhood satisfies the conclusion of the lemma with $X_i$.
\end{proof}

Given a set $A \subseteq \R^m$ and $i \in \{1,...,n\}$ we define
\begin{align*}
  A_i := \{p \in A : |A \cap (p \cdot \R X_i)| > 0\}.
\end{align*}
By $p \cdot \R X_i$, we mean the 1-dimensional $\R$-flow of the vector field $X_i$ that passes through $p \in \R^m$.  Note that this is an analytic submanifold.

Given a word $w$ written in the alphabet $\{1,...,n\}$, we define $A_w = (A_{w'})_i$ where $w = w'i$ and $A_\emptyset = A$.  Note that $A_w \subseteq A_{w'}$.  Let $w$ denote the word $123\cdots n$, the concatenation of all the letters.  Define the word $w^k$ to be the $k$-fold concatenation of $w$ (so $w^k$ is $kn$ letters long).

\begin{lemma}
  If $A \subset \R^m$ is a measure zero set, then $A_{w^m} = \emptyset$.
\end{lemma}
\begin{proof}
  Suppose otherwise.  There then exists a point $p \in A_{w^m} = (A_{w'})_n$ and so
  $$| A_{w'} \cap (p \cdot \R X_n)| > 0.$$
  We let $H_1$ denote the analytic manifold $p \cdot \R X_n$, which has dimension 1.  As $A_{w'} \subseteq A_{w^{m-1}}$, we get that $|A_{w^{m-1}} \cap H_1| > 0$.

  Now suppose we have a $k$-dimensional analytic manifold $H_k$ that intersects $A_{w^{m-k}}$ in a positive measure set (based on the surface area of $H_k$).  Thus, we can find a density point of $A_{w^{m-k}} \cap H_k$ satisfying the previous lemma, {i.e.}, there exists a density point $p$ of $A_{w^{m-k}} \cap H_k$, an open neighborhood $U \subseteq H_k$ of $ p $, and an index $i \in \{1,...,n\}$ so that $X_i \notin T_qH_k$ for all $q \in U$.
  
  As  $ A_{w^{m-k}} \sus (A_{w^{m-k-1}})_{1 \cdots (i+1)} $, by definition of the set $ (A_{w^{m-k-1}})_{1 \cdots (i+1)}  $ we have for any  $q \in A_{w^{m-k}} \cap H_k$ that
  \[
  |(A_{w^{m-k-1}})_{1 \cdots i} \cap (q \cdot \R X_i)| > 0.
  \]
  Since $(A_{w^{m-k-1}})_{1 \cdots i} \subseteq A_{w^{m-k-1}}$, we get for all $q \in A_{w^{m-k}} \cap H_k$ that
  \begin{align*}
    |A_{w^{m-k-1}} \cap (q \cdot \R X_i)| \geq   |(A_{w^{m-k-1}})_{1 \cdots i} \cap (q \cdot \R X_i)| > 0.
  \end{align*}
 Let $H_{k+1} = \bigcup_{q \in U} (q \cdot \R X_i)$.  As $X_i(q) \notin T_pU$, we get that $H_{k+1}$ is a $k+1$-dimensional analytic manifold and $|H_{k+1} \cap A_{w^{m-k-1}}| > 0$.

  We repeat until we get an $m$-dimensional analytic manifold for which $|H_m \cap A_\emptyset| = |H_m \cap A| > 0$.  But as $H_m$ has the same dimension as $\R^m$, this means $|A| > 0$, contradicting our assumption.
\end{proof}

\begin{proposition}\label{t:null-decomp}
  Let $A \subset \R^m$ have measure 0.  Then there exists a decomposition $A = C_1 \cup ... \cup C_n$ into Borel sets
  $$C_i = \bigcup_{k=0}^{m-1} (A_{w^k1\cdots (i-1)} \backslash A_{w^k1\cdots i})$$
  so that for each $i \in \{1,...,n\}$ we have
  \begin{align*}
    |C_i \cap (x \cdot \R X_i)| = 0, \qquad \forall x \in \R^m.
  \end{align*}
\end{proposition}

\begin{proof}
  Let $B_1 = \{p \in A_\emptyset : |A_\emptyset \cap (p \cdot \R X_1)| = 0\}$.  Then $A = B_1 \cup A_1$ where $A_1$ and $B_1$ are both Borel by Lemma~\ref{l:borel-decomp}, and we get that
  \begin{align*}
    |B_1 \cap (p \cdot \R X_1)| = 0, \qquad \forall p \in \R^m.
  \end{align*}
  By induction, we get a Borel decomposition
  $$A = B_1 \cup B_{12} \cup B_{123} \cup ... \cup B_{w^{m-1} 1 \cdots (n-1)} \cup A_{w^m} = B_1 \cup ... \cup B_{w^{m-1} 1 \cdots (n-1)}.$$
  Note that we have for every $B_{w'i}$ that
  \begin{align*}
    | B_{w'i} \cap (p \cdot \R X_i)| = 0, \qquad \forall p \in \R^m.
  \end{align*}
  We take $C_i = \bigcup_{k=0}^{m-1} B_{w^k 1 \cdots i}$ to finish the proof of the proposition.
\end{proof}

\begin{proposition} \label{p:S-sequence}
  Let $H$ be a subgroup of $G$ with a Carnot structure generated by $a_1,...,a_m \in V_1(G)$ and $A \subset G$ be Borel.  Then $\vol_H(H \cap gA) = 0$ for all $g \in G$ if and only if $A \in \mathcal{U}(\{a_1,...,a_m\})$.
\end{proposition}
\begin{proof}
  The backwards direction is clear by Fubini. We will prove the forwards direction. Let $H$ be homeomorphic to $\R^m$.  
  We will reuse the notation of the previous section where for each Borel set $E \subset G$ and word $w$, we define $E_{wi} = \{g \in E_w : |E_w \cap (g \cdot \R a_i)| > 0\}$. Lemma~\ref{l:borel-decomp} gives that these are Borel sets whenever $E$ is.  By construction, $E_{wi} \in \mathcal{U}(a_i)$. 
  
  We claim that $A = C_1 \cup ... \cup C_m$ where $C_i = \bigcup_{k=0}^{m-1} \left( A_{w^k1 \cdots (i-1)} \backslash A_{w^k1 \cdots i}\right)$, which proves the proposition.  Indeed, for any $g \in A$, we have by assumption that $\vol_H(H \cap g^{-1}A) = 0$.  As $e_H \in H \cap g^{-1}A$, by Theorem~\ref{t:null-decomp}, we have that there exists some $i$ so that
  $$e_H \in \bigcup_{k=0}^{m-1} \left( (H \cap g^{-1}A)_{w^k1 \cdots (i-1)} \backslash (H \cap g^{-1}A)_{w^k1 \cdots i}\right) \subseteq \bigcup_{k=0}^{m-1} \left( g^{-1}A_{w^k1 \cdots (i-1)} \backslash g^{-1}A_{w^k1 \cdots i}\right).$$
  This means that $g \in C_i$.
\end{proof}

\section{Differentiability of Lipschitz maps}
Suppose that $f \colon G \to H$ is a Lipschitz map between metric scalable groups. We remark that if $Df_g$ exists for some $ g\in G $, then it is Lipschitz as a function from $G $ to $ H$, with the same Lipschitz constant as $ f $.
\begin{lemma}
  Let $f : G \to \R$ be Lipschitz and $(N_m)_{m\in \N}$ be a filtration of $G$ by Carnot subgroups.  Then there exists a  $(N_m)_m$-negligible set $\Omega \subset G$ so that if $p \notin \Omega$ then for every $N_m$, the limit $\lim_{\lambda \to 0} \hat{f}_{p,\lambda}(u) $ exists for all $ u\in N_m $ and the resulting map on $ N_m $ is a homomorphism.
\end{lemma}
\begin{proof}
  Fix a $N_m$ and let $A_m$ denote the $p \in G$ for which the limit $\lim_{\lambda \to 0} \hat{f}_{p,\lambda}(u) $ does not exist or the limit map is not a homomorphism. We will show $\vol_{N_m}(gA_m \cap N_m) = 0$ for all $g$.  This would prove the lemma. 

  Now fix a $g \in G$ and let $p \in gA_m \cap N_m$.  If we define $F_g(u) = f(g^{-1}u)$ as a map defined on $N_m$, then $gp \in N_m$ is a nondifferentiability point of $F$.  However, by Pansu's theorem, $F$ is differentiable almost everywhere with respect to the Haar measure on $N_m$.  Thus, $\vol_{N_m}(gA_m \cap N_m) = 0$, which proves the lemma.
\end{proof}

We can now prove our differentiability result.

\begin{proof}[Proof of Theorem \ref{th:main}]
  As the theorem is vacuous if $G$ does not admit a filtration by Carnot subgroups, we may assume there is a filtration $(N_m)_m$.  By the previous lemma, we have that $Df_g$ exists and is a homomorphism when restricted to any $N_m$ for $g$ outside of a $(N_m)$-negligible set.  Take such a $g$.  We first claim that $Df_g$ exists on all of $G$.  Indeed, this follows from the fact that the maps
  \begin{align*}
    u \mapsto n(f(g \delta_{1/n}(u)) - f(g))
  \end{align*}
  are uniformly Lipschitz and converge, by assumption, on the dense subset $\bigcup_m N_m$.

  As $Df_g$ is a homomorphism when restricted to any $N_m$, an easy density argument then gives that $Df_g$ is also a homomorphism.  This proves the theorem.
\end{proof}

\section{Examples}\label{sec:ex}
\subsection{$L^p$- and $\ell_p$-sums}
We start with the definition of $L^p$-sums of metric spaces.
Let $\Omega=(\Omega, \mu)$ be a measure space, e.g., the natural numbers $\N$ with the counting measure.
Fix $p\in[1,\infty)$.
For each $\omega \in \Omega$ fix a pointed metric space $X_\omega=(X_\omega, d_\omega, \star_\omega)$.
We first define the collection 
${\mathcal M}(\Omega,  ( X_\omega)_\omega )$
of  {\em measurable sequences} as the set of those sequences $(x_\omega)_{\omega\in\Omega} $ with $  x_\omega \in X_\omega $ such that the function $   \omega \in \Omega \mapsto d  (   x_\omega, \star_\omega   )\in \R$ is measurable.
Then we define
$$L^p( ( X_\omega)_\omega) := 
 \{ (x_\omega)_{\omega } \in {\mathcal M}(\Omega,  ( X_\omega)_\omega ) :    
\int  d  (   x_\omega, \star_\omega   )^p {\rm d} \mu(\omega) < \infty\}.
$$
 We write $L^p(\Omega; X)$ for $L^p( ( X_\omega)_\omega)$ if $X_\omega=X$ for all $\omega \in \Omega$. 

The {\em distance function} on $L^p( ( X_\omega)_\omega)$ between  $(x_\omega)_{\omega\in\Omega}, 
(y_\omega)_{\omega\in\Omega}
\in L^p( ( X_\omega)_\omega)$
is 
$$ d  (  (x_\omega)_{\omega\in\Omega}, 
(y_\omega)_{\omega\in\Omega}   )    :=  \left( \int  d  (   x_\omega, y_\omega   )^p {\rm d} \mu(\omega)  \right)^{1/p}.$$

\begin{proposition} The set $L^p( ( X_\omega)_\omega) $ is naturally a pointed metric space, which is geodesic if all $X_\omega$ are geodesic.
\end{proposition}
\begin{proof}
	The fact that $ d $ is a metric for  $L^p( ( X_\omega)_\omega) $ follows from the usual proof of Minkowski inequality for the norm $ \|(x_\omega)_{\omega\in\Omega}\|_p \coloneqq  d  (  (x_\omega)_{\omega\in\Omega}, (\star_\omega)_{\omega\in\Omega}) $.
	
	Let us then show that $L^p( ( X_\omega)_\omega) $ is geodesic if $ X_\omega $ is geodesic for each $ \omega \in \Omega $. Let $ (x_\omega)_{\omega }, (y_\omega)_{\omega } \in L^p( ( X_\omega)_\omega) $. Now for all $ \omega \in \Omega $ there exists a curve $ \gamma_\omega \colon [0,1] \to X_\omega $ taking $ x_\omega $ to $ y_\omega $ such that $ d(x_\omega,y_\omega) = L(\gamma_\omega)$. We may assume that $ \gamma_\omega $ are parametrized by constant speed.
	
	Let $ \gamma \colon [0,1] \to L^p( ( X_\omega)_\omega) $, $ \gamma(t) = (\gamma_\omega(t))_\omega $. The curve $ \gamma $ is well defined, since for all $ t\in [0,1] $,
	\[
	d(\gamma(t),(x_\omega)_\omega)^p = \int d(\gamma_\omega(t),x_\omega)^p \mathrm{d}\mu(\omega) \leq \int d(y_\omega,x_\omega)^p \mathrm{d}\mu(\omega) = d((y_\omega)_\omega,(x_\omega)_\omega)^p 
	\]
	and so
	\[
	d(\gamma(t),(\star_\omega)_\omega) \leq d(\gamma(t),(x_\omega)_\omega) + d((x_\omega)_\omega,(\star_\omega)_\omega) \leq d((y_\omega)_\omega,(\star_\omega)_\omega)+ 2d((x_\omega)_\omega,(\star_\omega)_\omega) < \infty.
	\]
	
	Let then $ 0 = t_0 < t_1 < \ldots < t_n = 1 $ be a partition of $ [0,1] $. Since $ \gamma_\omega $ are geodesics with constant speed, we have that
	\[
	d(\gamma_\omega(t_{i-1}),\gamma_\omega(t_{i})) = (t_i-t_{i-1})d(x_\omega,y_\omega)
	\]
	for all $ \omega\in \Omega $ and $ i \in \{1,\ldots, n\} $. Therefore
	\begin{align*}
	\sum_{i=1}^{n}d(\gamma(t_{i-1}),\gamma(t_i)) &= 	\sum_{i=1}^{n} \left( \int d(\gamma_\omega(t_{i-1}),\gamma_\omega(t_i))^p \mathrm{d}\mu(\omega)\right)^{1/p} \\
	&=\sum_{i=1}^{n} \left( \int (t_i-t_{i-1})^pd(x_\omega,y_\omega)^p \mathrm{d}\mu(\omega)\right)^{1/p} \\
	&= \sum_{i=1}^{n}(t_i-t_{i-1} )\,d((x_\omega)_{\omega },(y_\omega)_{\omega }) \\
	&= d((x_\omega)_{\omega },(y_\omega)_{\omega }).
	\end{align*}
	Hence
	\[
	L(\gamma) = \inf_{\mathcal{P}}(\sum_{t_i \in \mathcal{P}} d(\gamma(t_{i-1}),\gamma(t_i))) =  d((x_\omega)_{\omega },(y_\omega)_{\omega }),
	\]
	where the infimum is over all partitions $ \mathcal P $ of $ [0,1] $. The proof is complete.
\end{proof}
Notice that if each $ X_\omega $ admits a group structure we may define a group operation for $ L^p( ( X_\omega)_\omega) $ element wise.
We focus now on $\ell_p$-sums of groups.
For a countable family $ \{G_n\}_{n\in \N} $ of groups we define the $\ell_p((G_n)_n)$ as  
$$\ell_p((G_n)_n) \coloneqq \{ (x_n)_{n\in\N} :   x_n \in G_n, 
\sum_{n\in\N}  d  (   x_n, e_n   )^p     < \infty\}, $$
\[
(x_n)_n \cdot (y_n)_n \coloneqq (x_ny_n)_n.
\]
We write $\ell_p(G)$ for $\ell_p((G_n)_n)$ if $G_n=G$ for all $n\in \N$. 

\begin{proposition}  Let $(G_n)_{n\in \N}$ be sequence of topological groups metrized by left-invariant metrics and let $ p\in [1,\infty) $.
Then $\ell_p( (G_n)_{n\in \N} )$ is a topological group.
\end{proposition}

\begin{proof}
We first show that the right translations are continuous. Fix $(b_n)_n\in \ell_p( (G_n)_{n\in \N} )$, that is,
$b_n\in G_n$ and $\sum_{n=1}^\infty |b_n|^p < \infty$, where $|b_n|:=d(b_n, e)$ and $d$ is the distance on $G_n$.
Let $(a_{n,j})_n$ be a sequence in $ \ell_p( (G_n)_{n\in \N} )$ converging to some $(a_{n})_n$.
Fix some $\eps>0$. We take $N$ large enough so that $\sum_{n=N+1}^\infty |b_n|^p < \eps$.
Then, being $N$ fixed and being the right translations $R_{b_1}, \ldots, R_{b_N}$ continuous, we take $J$ large enough so that for all $j>J$ and all $n=1,\ldots, N$
\begin{eqnarray}
d( (a_{n,j})_n  ,  (a_{n})_n  ) &<&\eps\\
d( R_{b_n}(a_{n,j})  ,  R_{b_n}(a_{n})  )^p &<&\eps/N.
\end{eqnarray}
Notice that consequently 
\begin{eqnarray*}
\sum_{n=N+1}^\infty  d (   a_{n,j} b_n  , a_{n} b_n  )^p 
&\leq & 
\sum_{n=N+1}^\infty  ( d (   a_{n,j} b_n  , a_{n,j}  )+ d (   a_{n,j}  , a_{n}    )+ d (   a_{n }   , a_{n} b_n  )   )^p
\\
&= & 
\sum_{n=N+1}^\infty  ( |b_n  |+ d (   a_{n,j}  , a_{n}    )+ | b_n  |   )^p
\\
&\leq & 
2^pd( (a_{n,j})_n  ,  (a_{n})_n  ) + 2^{p+1} \sum_{n=N+1}^\infty |b|  
\\
&\leq & 2^p\cdot3\eps,
\end{eqnarray*}
where we used the trick
\[
\sum(a+b)^p \leq \sum 2^p \max\{a,b\}^p \leq 2^p(\sum a^p + \sum b^p).
\]
Then we get for all $j>J$
\begin{eqnarray*}
d( R_{(b_n)_n}(a_{n,j})_n  , R_{(b_n)_n} (a_{n})_n  )
&= & 
\sum_{n= 1}^\infty d (   a_{n,j} b_n  , a_{n} b_n  )^p 
\\
&= & 
\left(\sum_{n= 1}^N   d( R_{b_n}(a_{n,j})  ,  R_{b_n}(a_{n})  )^p   \right)+\sum_{n=N+1}^\infty  d (   a_{n,j} b_n  , a_{n} b_n  )^p \\
&\leq & N \eps /N + 2^p\cdot 3 \eps = (1+3\cdot 2^p)\eps.
\end{eqnarray*}

Consequently, the multiplication in $\ell_p( (G_n)_{n\in \N} )$ is continuous since, if $(a_{n,j})_n \to (a_{n})_n$ and $(b_{n,j})_n \to (b_{n})_n$, as $j\to \infty$, then using left invariance we have
\begin{eqnarray*}
d( (a_{n,j})_n  (b_{n,j})_n ,   (a_{n})_n (b_n)_n  ) &\leq& 
d( (a_{n,j})_n  (b_{n,j})_n ,   (a_{n,j})_n (b_{n})_n  ) + d( (a_{n ,j})_n  (b_n)_n ,   (a_{n})_n (b_n)_n  )
\\
&\leq& 
d(    (b_{n,j})_n ,     (b_{n})_n  )  + d( R_{(b_n)_n}(a_{n,j})_n  , R_{(b_n)_n} (a_{n})_n  ) \to 0.
\end{eqnarray*}

We then show that the inversion is also continuous. Let  $(a_{n,j})_n \to (a_{n})_n$. Take $N$ large so that   $\sum_{n=N+1}^\infty |a_n|^p < \eps$.
Since the inversions in $G_1, \ldots, G_N$ are continuous, there exists $J$ such that for all $j>J$ and all $n=1,\ldots, N$ we have
\begin{eqnarray}
d( (a_{n,j})_n  ,  (a_{n})_n  ) &<&\eps\\
d( a_{n,j}^{-1}  ,  a_{n}^{-1}  )^p &<&\eps/N.
\end{eqnarray}
Then  for all $j>J$
\begin{eqnarray*}
d( (a_{n,j}^{-1} )_n   ,   (a_{n}^{-1} )_n    ) &=& 
\sum_{n= 1}^N  d (   a_{n,j} ^{-1}  , a_{n} ^{-1}   )^p
+
\sum_{n=N+1}^\infty  d (   a_{n,j} a_n^{-1}  , e  )^p
\\
&\leq & 
N \eps /N
+\sum_{n=N+1}^\infty  (  d (   a_{n,j} a_{n}^{-1} , a_{n,j}   ) + d ( a_{n,j} , a_{n}    )+ d (   a_n  , e  )   )^p
\\
&\leq & 
\eps + 
   2^p d( (a_{n,j})_n  ,  (a_{n})_n  ) + 2^{p+1}  \sum_{n=N+1}^\infty |a_n|^p 
\\
&\leq & (1+3\cdot 2^p)\eps.
\end{eqnarray*}
\end{proof}
\begin{remark}
	In a similar manner, we get that
	\[
    c_0 (G_n)_{n\in \N} ) \coloneqq \{ (x_n)_{n} \in \ell_\infty( (G_n)_{n\in \N} ) : \lim_{n\to \infty}d(x_n,e_{G_n}) = 0\}.
	\]
	is a topological group.
\end{remark}

\subsection{Examples of metric scalable groups} \label{s:examples}
Using the previous subsection, we can  build  examples of metric scalable groups starting with arbitrary sequences of Carnot groups equipped with homogeneous distances.
\begin{prop} \label{p:sum-scalable}
	Let $(G_n)_{n\in \N}$ be sequence of metric scalable groups and let $p \in [1,\infty)$. Then $\ell_p( (G_n)_{n\in \N} )$ is a metric scalable group. Moreover, if each $ G_n $ is complete and admits a filtration by Carnot subgroups, then $\ell_p( (G_n)_{n\in \N} )$ is complete and admits a filtration by Carnot subgroups.
\end{prop}
\begin{proof}
	We define the scaling map $ \delta \colon \R \times \ell_p((G_n)_n) \to \ell_p((G_n)_n)$ element wise using the scalings of each scalable group $ G_n $. By the previous proposition, $ \ell_p((G_n)_n) $ is a topological group. Hence it remains to see that $ \delta $ satisfies the conditions of a scalable group as in Definition \ref{def:scalable:group} and that the metric is homogeneous with respect to $ \delta $, which is straightforward to check. The proof for the fact that $\ell_p( (G_n)_{n\in \N} )$ is complete assuming that each $G_n$ is complete, is analogous to the proof of completeness of the classical $\ell_p$ spaces. Assume then that $(N_m^n)_m$ is a filtration by Carnot subgroups for each $G_n$. Then letting
	\[
	N_m = N_m^1 \times N_{m-1}^2 \times \cdots \times N_1^m \times \{e\}^\N
	\]
	for each $m\in \N$ defines a filtration by Carnot subgroups for $\ell_p( (G_n)_{n\in \N} )$. Indeed, each $N_m$ is isomorphic to a finite product of Carnot groups, and the union $\cup_m N_m$ is dense in $\ell_p( (G_n)_{n\in \N} )$ as the set of finite sequences is dense in $\ell_p( (G_n)_{n\in \N} )$.
\end{proof}
Proposition \ref{p:sum-scalable} gives us a simple way to construct many different noncommutative and infinite-dimensional metric scalable groups that admit filtrations by Carnot subgroups. Indeed, we may consider examples where each $ G_n $ is a Carnot group, like $G_n = \mathbb H^1$ or $G_n = \mathbb{H}^n$, where $ \mathbb{H}^n $ is the $ n $-th Heisenberg group equipped with a homogeneous distance.  
We stress that the last result does not require any bound on the nilpotency step of $(G_n)_{n\in \N}$, in case they are  Carnot groups.
In fact, an interesting example is when $G_n $ is the free Carnot group of step $n$ and rank 2, which we denote by $\mathbb F_{2,n}$. We state this example as a result.
\begin{proposition}
Even though $\ell_2((\mathbb F_{2,n})_n)$ is not nilpotent, 	it is a metric scalable group that is complete and admits a filtration by Carnot groups. Moreover, the subset $V_1(\ell_2((\mathbb F_{2,n})_n))$ generates $\ell_2((\mathbb F_{2,n})_n)$ and is separable.
\end{proposition}
\begin{proof}
The space $\ell_2((\mathbb F_{2,n})_n)$ is a metric scalable group by Proposition \ref{p:sum-scalable} and the filtration is simply given by
\[
N_m = \mathbb{F}_{2,1} \times \cdots \times \mathbb{F}_{2,m}.
\]
The first layer $V_1(\ell_2((\mathbb F_{2,n})_n))$ is now given by $\ell_2((V_1(\mathbb F_{2,n}))_n)$ as we defined the dilation map on $\ell_2((\mathbb F_{2,n})_n)$ component wise. Indeed, a sequence $(x_n)_n \in \ell_2((\mathbb F_{2,n})_n)$ is a one-parameter subgroup if and only if each $x_n \in \mathbb F_{2,n}$ is a one-parameter subgroup. The fact that $V_1(\ell_2((\mathbb F_{2,n})_n))$ generates follows from Proposition \ref{exist:filtration}. Moreover, $V_1(\ell_2((\mathbb F_{2,n})_n)) = \ell_2((V_1(\mathbb F_{2,n}))_n)$ is separable since now each $V_1(\mathbb F_{2,n})$ is separable.
\end{proof}
If $ \mathbb{H}^1 $ is the first Heisenberg group, then by Proposition \ref{p:sum-scalable}, for all $ p\in [1,\infty) $, the space $ \ell_p(\mathbb{H}^1) $ is a metric scalable group admitting filtration by Carnot subgroups. The space $ \ell_2(\mathbb{H}^1) $ has the extra property of being a Banach Lie group. Indeed, it can be modelled on $ \ell_2(\R^2) + \ell_1(\R) $, following \cite{Magnani-Rajala}.  However, we shall show that there are metric scalable groups, e.g. $ \ell_1(\mathbb{H}^1) $, admitting filtrations by Carnot groups that are not Banach manifolds. Hence the notion of metric scalable group strictly extends the one of Banach homogeneous group as defined in \cite{Magnani-Rajala}.
\begin{proposition}
The topological group $\ell_1( \mathbb H^{1})$ is not a Banach Lie group.
\end{proposition}
\begin{proof}
	Suppose by contradiction that $\ell_1( \mathbb H^{1})$ is a Banach Lie group and let $ Z $ be the center of $ \mathbb H $. As $ Z $ is a closed subgroup of a Banach Lie group, $ Z $ is a Banach Lie group as well. However, recall that  $ (Z,d_\mathbb H) $ is isometric to $(\R,\sqrt{d_E}) $ and therefore
	\[
	 \ell_1(Z) = \{ (a_n)_{n\in \N}\,:\, \sum \sqrt{|a_n|} < \infty \} = \ell_{1/2}(\R).
\]
 Hence also $ \ell_{1/2}(\R) $ has a structure of a Banach Lie  group and its Lie algebra is a Banach space. But since $ \ell_{1/2}(\R) $ is a vector space, the exponential map $ \exp \colon \Lie(\ell_{1/2}(\R)) \to \ell_{1/2}(\R) $ is a linear isomorphism. This is a contradiction since $ \ell_{1/2}(\R) $ is not locally convex, and therefore not a normed space.
\end{proof}

\appendix
\section{Some useful commutator identities}
\label{ss:commidentities}
\begin{lemma}
	\label{l:split2}
	Let $ G $ be a group and $ x,y,z\in G $. Then
	\[
	[xy,z] = [x,[y,z]][y,z][x,z] \ja [z,xy] = [z,x][z,y][[y,z],x] = h[z,x][z,y],
	\]
	where $ h $ is a product of commutators of $ x,y,z $ of weight $ \geq 3 $.
\end{lemma}
\begin{proof}
	For the first equation,
	\[
	[y,z][x,z] = [y,z]xz\inv{x}\inv{z} = [[y,z],x]xyz\inv{y}\inv{z}z\inv{x}\inv{z} = [[y,z],x][xy,z].
	\]
	Since  $ [a,b]=\inv{[b,a]} $, we get
	\[
	[xy,z] = [x,[y,z]][y,z][x,z].
	\]
	Using this we get again by $ [a,b]=\inv{[b,a]} $ that
	\[
	[z,xy] = [z,x][z,y][[y,z],x].
	\]
	The last equation follows by reordering the terms, which produces some higher order commutators into $ h $.
\end{proof}
\begin{corollary}
	\label{splitlemma}
	If  $ [y,z]\in Z(G) $, then
	\[
	[xy,z] = [x,z][y,z] \ja [z,xy] = [z,x][z,y].
	\]
\end{corollary}
\begin{corollary}
	\label{cor:split}
	Let $ n,m\in \N $. We have that
	\[
	[x^n,y^m] = h[x,y]^{nm},
	\]
	where $ h $ is a product of commutators of $ x $ and $ y $ of weight $ \geq 3 $.
\end{corollary}
\begin{proof}
	The proof is by iterating Lemma~\ref{l:split2} for $ nm $ times and reordering the terms, which produces some additional higher order commutators into $ h $.
\end{proof}
\begin{lemma}
	\label{inverse}
	Let $ G $ be a group, $ x,y\in G $. Then
	\[
	[\inv{x},y] = [\inv{x},[y,x]]\inv{[x,y]}.
	\]
\end{lemma}
\begin{proof}
	The statement follows from
	\[
	[\inv{x},[y,x]] = \inv{x}yx\inv{y}\inv{x}x\inv{[y,x]} = [\inv{x},y][x,y].
	\]
\end{proof}

\newpage
\bibliography{general_bibliography} 
 
\bibliographystyle{amsalpha}

\end{document}